\documentclass{article}
\usepackage{arxiv}

\usepackage[utf8]{inputenc} 
\usepackage[T1]{fontenc}    
\usepackage{hyperref}       
\usepackage{url}            
\usepackage{booktabs}       
\usepackage{amsfonts}       
\usepackage{nicefrac}       
\usepackage{microtype}      
\usepackage{lipsum}
\usepackage{mathtools,physics}
\usepackage{enumerate}
\usepackage{graphicx}
\graphicspath{ {./images/} }
\usepackage{amsmath}
\usepackage{amssymb}
\usepackage{amsthm}
\usepackage{stackengine} 
\stackMath
\usepackage[framed,autolinebreaks,useliterate]{mcode} 
\usepackage{centernot}
\usepackage{soul} 
\usepackage{centernot}
\usepackage{authblk}

\newtheorem{theorem}{Theorem}[section]
\newtheorem{lemma}[theorem]{Lemma}
\newtheorem{corollary}[theorem]{Corollary}
\newtheorem{remark}[theorem]{Remark}
\newtheorem{example}[theorem]{Example}

\newtheorem{definition}[theorem]{Definition}

\newcommand{\rng}[1]{\mathcal{C}(#1)}
\newcommand{\nullsp}[1]{\mathcal{N}(#1)}
\newcommand{\pinv}[1]{{#1}^+\!}

\newcommand{\vect}[1]{\operatorname{vec}\left({#1}\right)}

\DeclarePairedDelimiter\brac{(}{)}

\newcommand{\bK}{\mathbb{K}}
\newcommand{\bN}{\mathbb{N}}
\newcommand{\bC}{\mathbb{C}}
\newcommand{\bR}{\mathbb{R}}
\newcommand{\cV}{\mathcal{V}}

\newcommand{\diag}{\operatorname{diag}}

\makeatletter
\let\save@mathaccent\mathaccent
\newcommand*\if@single[3]{%
  \setbox0\hbox{${\mathaccent"0362{#1}}^H$}%
  \setbox2\hbox{${\mathaccent"0362{\kern0pt#1}}^H$}%
  \ifdim\ht0=\ht2 #3\else #2\fi
  }
\newcommand*\rel@kern[1]{\kern#1\dimexpr\macc@kerna}
\newcommand*\widebar[1]{\@ifnextchar^{{\wide@bar{#1}{0}}}{\wide@bar{#1}{1}}}
\newcommand*\wide@bar[2]{\if@single{#1}{\wide@bar@{#1}{#2}{1}}{\wide@bar@{#1}{#2}{2}}}
\newcommand*\wide@bar@[3]{%
  \begingroup
  \def\mathaccent##1##2{%
    \let\mathaccent\save@mathaccent
    \if#32 \let\macc@nucleus\first@char \fi
    \setbox\z@\hbox{$\macc@style{\macc@nucleus}_{}$}%
    \setbox\tw@\hbox{$\macc@style{\macc@nucleus}{}_{}$}%
    \dimen@\wd\tw@
    \advance\dimen@-\wd\z@
    \divide\dimen@ 3
    \@tempdima\wd\tw@
    \advance\@tempdima-\scriptspace
    \divide\@tempdima 10
    \advance\dimen@-\@tempdima
    \ifdim\dimen@>\z@ \dimen@0pt\fi
    \rel@kern{0.6}\kern-\dimen@
    \if#31
      \overline{\rel@kern{-0.6}\kern\dimen@\macc@nucleus\rel@kern{0.4}\kern\dimen@}%
      \advance\dimen@0.4\dimexpr\macc@kerna
      \let\final@kern#2%
      \ifdim\dimen@<\z@ \let\final@kern1\fi
      \if\final@kern1 \kern-\dimen@\fi
    \else
      \overline{\rel@kern{-0.6}\kern\dimen@#1}%
    \fi
  }%
  \macc@depth\@ne
  \let\math@bgroup\@empty \let\math@egroup\macc@set@skewchar
  \mathsurround\z@ \frozen@everymath{\mathgroup\macc@group\relax}%
  \macc@set@skewchar\relax
  \let\mathaccentV\macc@nested@a
  \if#31
    \macc@nested@a\relax111{#1}%
  \else
    \def\gobble@till@marker##1\endmarker{}%
    \futurelet\first@char\gobble@till@marker#1\endmarker
    \ifcat\noexpand\first@char A\else
      \def\first@char{}%
    \fi
    \macc@nested@a\relax111{\first@char}%
  \fi
  \endgroup
}
\makeatother
\title{The Reverse Order Law and the Riccati Equation}

\author[1,2]{
  Oskar Kędzierski
}

\affil[1]{Instytut Matematyki UW, Warsaw, Poland \authorcr
  \texttt{oskar@mimuw.edu.pl}
}

\affil[2]{
  NASK National Research Institute, Warsaw, Poland \authorcr
  \texttt{oskar.kedzierski@nask.pl}
}

\begin{document}
\maketitle

\begin{abstract}
We give a full analytic solution to a particular case of the algebraic Riccati equation $XWW^*WX=W^*$ for any matrix $W$ (possibly non-square or non-symmetric) in using the Schur method, terms of the SVD decomposition of $W$. In particular, $(WX)^3=WX$ and $(XW)^3=XW$ for any solution $X$. We show that for $W=AB$, matrix $X=\pinv{B}\pinv{A}$ is a solution of this equation if and only if the reverse order law holds, i.e., $\pinv{(AB)}=\pinv{B}\pinv{A}$. For a Hermitian and invertible $W$ the maximal and stabilizing Hermitian solutions is shown to be equal to $\pinv{W}$. Equivalence to the equation $XWX=\pinv{W}$ is proven.


\end{abstract}

\section{Introduction}
Greville proved in~\cite{Greville_note} that for any compatible complex matrices $A,B$
  \begin{equation}\label{eq:greville}
 \pinv{\brac*{AB}}=\pinv{B}\pinv{A}\Longleftrightarrow \pinv{A}ABB^*A^*AB\pinv{B}=BB^*A^*A.    
 \end{equation}
By premultiplyig the second equation by $\pinv{B}$, and by $\pinv{A}$ on the right, one gets
that $X=\pinv{B}\pinv{A}$ satifies a particular case of the algebraic Ricatti equation~\cite{riccati_book}
\begin{equation}\label{eq:riccati}
XWW^*WX=W^*,    
\end{equation}
for $W=AB$.
By premultiplying equation~(\ref{eq:riccati}) by $XW$ on the left and by $WX$ on the right, one sees that for any $N\in\bN$
\[(XW)^NW^*(WX)^N=W^*,\]
 which suggests some periodicity of matrices $XW$ and $WX$. We solve analytically the equation~(\ref{eq:riccati}) for any matrix $W\in\bK^{m\times n}$
 using the Schur method as follows. If $W=U\Sigma V^*$ is any SVD decomposition of $W$ over $\bK$ then $X$ satisfies~(\ref{eq:riccati}) if and only if 
  \[X=V\pinv{\Sigma}DU^*+P_{\nullsp{W}} Y P_{\nullsp{W^*}},\]
  for an arbitrary matrix $Y\in\bK^{n\times m}$ and a particular block diagonal matrix $D=\diag(X_1,\ldots,X_t)$, where $X_i$ are arbitrary square matrices over $\bK$,
  of sizes equal to the multiplicities of the singular values of $W$ such that $X_i^2=I$.
 By a result of Tian, we show that the following generalization of~(\ref{eq:greville})
holds, i.e., 
 \[\pinv{\brac*{AB}}=\pinv{B}\pinv{A}\Longleftrightarrow \brac*{\pinv{B}\pinv{A}}WW^*W\brac*{\pinv{B}\pinv{A}}=W^*,\]
 with $W=AB$, that is, $\pinv{\brac*{AB}},\pinv{B}\pinv{A}$ solve both equation~(\ref{eq:riccati}) for $W=AB$ if and only if they are equal.

 I would like to thank Michał Karpowicz, who observed originally the periodicity and attracted my attention to the Riccati equation.

\section{Notation}
The set of all matrices with $m$ rows, $n$ columns, and coefficients in field $\bK$ is denoted $\bK^{m\times n}$, where $\bK=\bR$ or $\bK=\bC$.  The column space of matrix $A$ is denoted by $\rng{A}$ and the null space by $\nullsp{A}$.  The Moore--Penrose pseudoinverse of matrix $A$ is denoted by $\pinv{A}$. The complex conjugate is denoted by $A^*$. Matrix is Hermitian (or symmetric over $\bR$) if $A^*=A$. A square matrix is unitary (or orthogonal over $\bR$) if $A^*A=I$ where $I$ denotes the unit matrix. SVD stands for singular value decomposition, i.e., decomposition $A=U\Sigma V^*$, where $U,V$ are unitary matrices, and $\Sigma$ is a real non--negative generalized diagonal matrix with decreasing diagonal entries. If $\rank A=r$ then the first $r$ columns of matrix $U$ are called left singular vectors of $A$ and likewise the first $r$ columns of matrix $V$ are called right singular vectors of $A$. For any matrix $A\in\bK^{m\times n}$, the matrix $A_{p:q,s:t}$ is a submatrix of $A$ consisting of rows from $p$ to $q$ and columns from $s$ to $t$, both $q$ and $t$ included (MATLAB notation). A single colon with no numbers indicates that there are no restrictions. For two Hermitian matrices of the same size we write $A\preceq B$ if and only if $B-A$ is positive semidefinite. 

More details on pseudoinverses and generalized inverses can be found in~\cite{BenGreville},\cite{WangWeiQiao} and \cite{Bernstein}.

\section{Preliminaries}
\begin{lemma}
Let $W\in\bK^{m\times n}$. Then any eigenvector $w\in\bK^n$ of the matrix, 
\[M=\left[\begin{array}{c|c}
        0 & WW^*W\\ \hline
        W^* &  0\\
    \end{array}
    \right],\]
    which is not in the kernel of $M$, is of the form
    \[w=\left[\begin{array}{c}
        u\\ \hline \pm\frac{1}{\sigma_i}v
    \end{array}
    \right]\in\bK^{m+n},\] 
    where $\sigma\in\bR$ is a singular value of $W$. In addition, $u$ is a left singular vector of $W$ and $v$ is the corresponding right singular vector of $W$, that is, $W^*u=\sigma v$.
\end{lemma}

\begin{proof}
    Let $r=\rank W$ and let $W=U\Sigma V^*$ be an SVD decomposition of matrix $W$ over $\bK$. Then $WW^*W=U\Sigma\Sigma^*\Sigma V^*$.
    Assume that $w=\left[\begin{array}{c}
        w_1 \\ \hline w_2
    \end{array}
    \right]$ is an eigenvector of matrix $M$, that is $Mw=\lambda w$ for some $\lambda\in\bK$ and $w\neq 0$.  Then
  \[\left[\begin{array}{c|c}
        0 & U\Sigma\Sigma^*\Sigma V^*\\ \hline
        V^*\Sigma^* U &  0\\
    \end{array}
    \right]
    \left[\begin{array}{c}
        w_1 \\ \hline w_2
    \end{array}
    \right]=\lambda
    \left[\begin{array}{c}
        w_1 \\ \hline w_2
    \end{array}
    \right],
    \]
 implies that
 \[\left\{
 \begin{array}{ccc}
    U\Sigma\Sigma^*\Sigma V^* w_2 & = & \lambda w_1,\\
     V \Sigma^* U^* w_1 & = & \lambda w_2. 
 \end{array}
 \right.\]
 Assume that $\lambda\neq 0 $ and substitute the second equation to the first one, 
 \[U(\Sigma\Sigma^*)^2  U^*w_1=\lambda^2 w_1,\]
 where $(\Sigma\Sigma^*)^2=\diag(\sigma_1^4,\ldots,\sigma_r^4,0,\ldots,0)\in\bK^{m\times m}$.
 That is $\lambda=\pm \sigma^2$, where $\sigma=\sigma_i$ for some $i=1,\ldots,r$, and $w_1=u$ is some left singular vector of the matrix $W$ corresponding to the singular value $\sigma$ (it is not necessarily a multiple of a column of matrix $U$ when $\sigma$ is a multiple singular value). From the second equation 
 \[V \Sigma^* U^* u  =  \pm\sigma^2 w_2 \Longrightarrow w_2=\pm\frac{1}{\sigma}v.\]
\end{proof}

\begin{lemma}\label{lem:diagonalizability}
Let $W\in\bK^{m\times n}$. Then matrix
\[M=\left[\begin{array}{c|c}
        0 & WW^*W\\ \hline
        W^* &  0\\
    \end{array}
    \right],\]
    is diagonalizable over $\bK$.
\end{lemma}

\begin{proof}
    Let $r=\rank W$ and let $W=U\Sigma V^*$ be an SVD decomposition of matrix $W$ over $\bK$. Let $u_1,\ldots,u_r\in\bK^m$ be the corresponding left singular vectors of $W$, that is, first $r$ columns of matrix $U$. Similarly, let $v_1,\ldots,v_r\in\bK^n$ be the corresponding right singular vectors of $W$, that is, the first $r$ columns of $V$.
    Then the $2r$ vectors
    \[\left[\begin{array}{c}
        u_i \\ \hline \pm\frac{1}{\sigma_i}v_i
    \end{array}
    \right]\in\bK^{m+n},\] 
    are eigenvectors of $M$ corresponding to the eigenvalue $\pm\sigma_i^2$, as $Wv_i=\sigma_i u_i$ and $W^*u_i=\sigma_i v_i$.
    for $i=1,\ldots,\rank W$. The null space of $M$ contains vectors of the form
\[w=\left[\begin{array}{c}
        w_1 \\ \hline w_2
    \end{array}
    \right]\in\bK^{m+n},\quad\text{where}\quad w_1\in\nullsp{W^*}, w_2\in\nullsp{W}.\]
    This gives in total $2r+(m-r)+(n-r)=m+n$ linearly independent vectors and $M$ is an $(m+n)$-by-$(m+n)$ matrix. Therefore $\nullsp{M}=\nullsp{W^*}\times\nullsp{W}$ and $M$ is diagonalizable.
\end{proof}

\begin{lemma}\label{lem:decomp}
     Let $W\in\bK^{m\times n}$ and let $\cV\subset \bK^{m+n}$ be an invariant $m$-dimensional subspace of matrix $M$,
     such that the projection $\pi\colon\bR^m\times\bR^n\rightarrow\bR^m$ is surjective when restricted to $\cV$. Then for any $i=1,\ldots,t$
     \[\dim\cV_{(\sigma_i^2)}+\dim\cV_{(-\sigma_i^2)}=\dim V_{WW^*,(\sigma_i)}.\]
     Moreover, $\dim\cV_{(0)}=m-r$.
\end{lemma}

\begin{proof}
    Restriction of a diagonalizable matrix to an invariant subspace is diagonalizable, cf.~\cite[Chapter~XIV, Ex.~13]{LangAlgebra} (minimal polynomial of an endomorphism is divisible by a minimal polynomial of its restriction to an invariant subspace). Therefore, $\cV$ as a direct sum of eigenspaces (a priori, some possibly equal to 0)
    \[\cV=\cV_{(\sigma_1^2)}\oplus \cV_{(-\sigma_1^2)}\oplus \cdots\oplus \cV_{(\sigma_t^2)}\oplus \cV_{(-\sigma_t^2)}\oplus \cV_{(0)},\]
    where $\sigma_1,\ldots,\sigma_t\in\bR_{>0}$ are (unlike in Lemma~\ref{lem:diagonalizability}) pairwise distinct singular values of matrix $W$. and let $s_i=\dim V_{WW^*,(\sigma_i)}$ be the multiplicity of $\sigma_i$ for $i=1,\ldots,t$.
    By definition (cf. the proof of Lemma~\ref{lem:diagonalizability}) and the surjectivity assumption
    \[\pi\brac*{\cV_{(\sigma_i^2)}\oplus \cV_{(-\sigma_i^2)} }=V_{WW^*,(\sigma_i)},\]
    \[\pi\brac*{\cV_{(0)}}=V_{WW^*,(0)}=\nullsp{W^*}.\]
    The dimension of the image of a subspace does not increase, therefore
    \[\dim {\cV_{(\sigma_i^2)}\oplus \cV_{(-\sigma_i^2)}}\ge \dim V_{WW^*,(\sigma_i)},\]
    \[\dim \cV_{(0)}\ge \dim V_{WW^*,(0)}.\]
    Neither of those inequalities may be strict as both sides sum to $m$. 
\end{proof}

The following lemma is well--known, cf.~\cite[Equation~(8.23)]{Tian512}.
\begin{lemma}\label{lem:equal_ranks}
For any matrices $A\in\bK^{m\times k},B\in\bK^{k\times n}$
\[\rank(\pinv{B}\pinv{A})=\rank(AB)=\rank\brac*{{\brac*{AB}}^*}=\rank\brac*{\pinv{\brac*{AB}}}.\]
\end{lemma}
\begin{proof}
Since $A^*=A^+AA^*$ and $B^*=B^*BB^+$ (for example, by the SVD decomposition)
    $B^*A^*=B^*B\brac*{B^+A^+}AA^*$,
    and therefore $\rank(B^*A^*)\le\rank(B^+A^+)$.
    Similarly 
    \[\pinv{A}=A^*\pinv{\brac*{A^*}}\pinv{A},\quad \pinv{B}=\pinv{B}\brac*{\pinv{B}}^*B^*\quad\text{and}\quad
    \pinv{B}\pinv{A}=\pinv{B}\brac*{\pinv{B}}^*\brac*{B^*A^*}\pinv{\brac*{A^*}}\pinv{A}\]
    so $\rank(B^*A^*)\ge\rank(B^+A^+)$.
\end{proof}

\section{Schur Trick}
\begin{lemma}\label{lem:solutions}
    
    Let $W\in\bK^{m\times n}$. Then $X\in\bK^{n\times m}$ solves the equation~(\ref{eq:riccati}) if and only if
    
    \begin{equation}\label{eq:schur}
\left[\begin{array}{c|c}
        0 & WW^*W\\ \hline
        W^* &  0\\
    \end{array}
    \right]
    \left[\begin{array}{c}
        I \\ \hline X
    \end{array}
    \right]=
    \left[\begin{array}{c}
        I \\ \hline X
    \end{array}
    \right]
    WW^*WX.
    \end{equation}
    
Moreover, solutions of equation~(\ref{eq:riccati}) are in one-to-one correspondence with $m$-dimensional invariant subspaces $\cV\subset\bK^{(m+n)\times(m+n)}$ of matrix $M=\left[\begin{array}{c|c}
        0 & WW^*W\\ \hline
        W^* &  0\\
    \end{array}
    \right]$ such that $\pi(\cV)=\bK^m$ under the projection $\pi\colon\bK^m\times\bK^m\rightarrow \bK^m$.
\end{lemma}

\begin{proof}
    Only the second claim requires a proof. If $X$ is a solution, then by equation~(\ref{eq:schur}) then $\cV$ is given by the column space of matrix $\left[\begin{array}{c}
        I \\ \hline X
    \end{array}
    \right]$. On the other hand let $\cV$ be an invariant subspace of $M$ such that $\pi(\cV)=\bK^m$. Fix a basis of $\cV$ and arrange itin columns of matrix $Z=\left[\begin{array}{c}
        V_1 \\ \hline 
        V_2
    \end{array}
    \right]\in\bK^{(m+n)\times n}$. Then, by the assumption, $V_1$ is invertible and the matrix $\left[\begin{array}{c}
        I \\ \hline 
        V_2 V_1^{-1}
    \end{array}\right]$ has the same invariant columns space as $Z$, that is $X=V_2V_1^{-1}$. Different choice of basis of $\cV$ leads
    to a matrix $
    ZC=\left[\begin{array}{c}
        V_1C \\ \hline 
        V_2C
    \end{array}
    \right]$, for some invertible $C\in\bK^{m\times m}$ which induces the same solution $X$.
\end{proof}

\section{Main Result}
Note that $X=\pinv{W}$ is always a solution of~(\ref{eq:riccati}).

\begin{theorem}\label{thm:main}
    Matrix $X\in\bK^{n\times m}$ is a solution of the equation~(\ref{eq:riccati}) if and only if there exist a block diagonal matrix
    \[D=\diag(X_1,\ldots,X_t,0,\ldots,0)\in\bK^{n\times m}\]
    such that $X_i\in\bK^{s_i\times s_i}$ if a matrix such that $X_i^2=I_{s_i}$ for $i=1,\ldots,t$ where $s_i=\dim V_{WW^*,(\sigma_i)}$, and there exists a matrix
    $Y\in\bK^{n\times m}$ such that
    \begin{equation}\label{eq:soln}
       X=V\pinv{\Sigma}DU^*+P_{\nullsp{W}} Y P_{\nullsp{W^*}}. 
    \end{equation}
       
\end{theorem}

\begin{proof}
    By Lemma~\ref{lem:solutions}, any solution correspond to an invariant subspace $\cV\subset\bK^{m+n}$ of matrix $M$, and by Lemmas~\ref{lem:diagonalizability} and~\ref{lem:decomp} there exists an SVD
    decomposition of $W=U\Sigma V^*$ over $\bK$ such that $U$ and $V^*$ induce a basis of $\cV$. 

    Recall that $s_i$ denotes the multiplicity of the singular value $\sigma_i$. Let $C=\diag(C_1,\ldots,C_t)\in\bK^{r\times r}$ be a block diagonal matrix with $C_i\in\bK^{s_i\times s_i}$ an invertible matrix and let $C_0\in\bK^{(m-r)\times (m-r)}$ be any invertible matrix. Let $D_i$ be a diagonal $s_i$-by-$s_i$ matrix with 1's and -1's on the diagonal, corresponding to the choice between $\sigma_i^2$ or $-\sigma_i^2$ eigenvalue of $M$. Let $Y\in\bK^{(n-r)\times (m-r)}$ be any matrix. Without losing generality one can replace $Y$ with $YC_0$. Let $V_0=V_{:,(r+1):n}$ and let $U_0=U_{:,(r+1):m}$. Then
\renewcommand{\arraystretch}{1.5}
     \[Z=\left[\begin{array}{c|c|c|c|c}
        U_1 C_1 & U_2 C_2 & \cdots & U_t C_t & U_0 C_0 \\ \hline 
        \frac{1}{\sigma_1}V_1 C_1 D_1 & \frac{1}{\sigma_2}V_2 C_2 D_2 & \cdots & \frac{1}{\sigma_t}V_t C_t D_t & V_0 YC_0\\ 
    \end{array}
    \right],\]
where , for $t_{-1}=0, t_i=s_1+\ldots+s_i$ for $i\ge 1$, the matrix $U_i=U_{:,(t_{i-1}+1):t_i}\in\bK^{m\times s_i}$ consists of left singular vectors of $W$, corresponding to the singular value $\sigma_i$, and similarly for $V_i\in\bK^{n\times s_i}$ for the left singular vectors, or in a more compact way, with $\widetilde{D}=\diag(D_1,\ldots,D_t)$
\begin{equation}\label{eq:eigenspace}
Z=\left[\begin{array}{c|c}
        U_{:,1:r}C & U_0 C_0 \\ \hline 
        V_{:,1:r}\pinv{\Sigma}C\widetilde{D} & V_0 YC_0\\ 
    \end{array}
    \right].
\end{equation}

Since 
    \[Z\left[\begin{array}{c|c}
        U_{:,1:r}C & U_0 C_0 
    \end{array}\right]^{-1}=Z\left[\begin{array}{c}
        C^{-1}U_{:,1:r}^* \\ \hline C_0^{-1} U_0^*  
    \end{array}\right]=\left[\begin{array}{c}
        I \\ \hline X
    \end{array}
    \right],\]
    we get

    \[X=V_{:,1:r} \diag\brac*{\frac{1}{\sigma_1} I_{s_1},\ldots,\frac{1}{\sigma_t} I_{s_t}}  C\widetilde{D} C^{-1}U_{:,1:r}^* + V_0 Y U_0^*, \]
or with $X_i=C_i D_i C_i^{-1}$
    \[X=V_{:,1:r} \diag\brac*{\frac{1}{\sigma_1} X_1,\ldots,\frac{1}{\sigma_t} X_i} U_{:,1:r}^* + V_0 Y U_0^*. \]

Equivalently, by replacing $Y$ with $V_0^* Y U_0$

 \[X=V \diag\brac*{\frac{1}{\sigma_1} X_1,\ldots,\frac{1}{\sigma_t} X_i,0,\ldots,0} U^* + V_0V_0^* YU_0 U_0^*,\]
 which is the same as~(\ref{eq:soln}).
 It is immediate to check that for any choice of matrices $X_i$'s such that $X_i^2=I_{s_i}$ and any $Y$ matrix given by~(\ref{eq:soln}) is a solution of~(\ref{eq:riccati}).
\end{proof}

\begin{corollary}\label{cor:equivalent}
    Let $X$ be a solution of equation~(\ref{eq:riccati}). Then
    \[XWX=\pinv{W}.\]
\end{corollary}

\begin{proof}
By Theorem~\ref{thm:main}, any solution is of the form~(\ref{eq:soln}).
Note that
    \[WP_{\nullsp{W}}YP_{\nullsp{W^*}}=0,\]
    \[P_{\nullsp{W}}YP_{\nullsp{W^*}}W=0,\]
    as $\rng{A}$ is orthogonal to $\nullsp{A^*}$ for any matrix $A$.
    Therefore
    \[XWX=(V\pinv{\Sigma}DU^*)U\Sigma V^*(V\pinv{\Sigma}DU^*)=V\pinv{\Sigma}U^*=\pinv{W},\]
    as diagonal matrices $\frac{1}{\sigma_i}I_{s_i}$ are in the center of each block.
\end{proof}

\begin{corollary}\label{cor:tripotent}
    Let $X$ be a solution of equation~(\ref{eq:riccati}). Then
    \[(WX)^2=P_{\rng{W}},\quad (XW)^2=P_{\rng{W^*}},\]
    \[(WX)^3=WX,\quad (XW)^3=XW.\]
\end{corollary}

\begin{proof}
By Corollary~\ref{cor:equivalent}
\[(WX)^2=W(XWX)=W\pinv{W},\quad (XW)^2=(XWX)W=\pinv{W}W.\]
By the previous line
\[(WX)^3=P_{\rng{W}}WX=WX,\quad (XW)^3=P_{\rng{W^*}}(XW)=XW,\]
as $\rng{X}\subset\rng{W^*}$.
\end{proof}

\begin{remark}
     When all singular values of $W$ are simple then the formula simplifies to
    \[X=V\pinv{\Sigma}DU^*+P_{\nullsp{W}} Y P_{\nullsp{W^*}},\]
    where $D=\diag(\pm 1,\ldots,\pm 1,0,\ldots,0)$.
\end{remark}

\begin{remark}
    Note the incidental similarity of Equation~\ref{eq:soln} to Tian's~\cite[Theorem~3.1(b), Eq.~(3.8)]{Tian512} for
    the $\{1,3,4\}$-inverse of $W$.
\end{remark}

\section{The Reverse Order Law}
In this section, it is investigated when $\pinv{B}\pinv{A}$ satisfies~(\ref{eq:riccati}) for $W=AB$. It turns out this is equivalent to the reverse order law.

The following Lemma characterizes solutions of~~(\ref{eq:riccati}) of the lowest rank, equal to $\rank W$. For the Penrose conditions/equations see~\cite[Chapter~1.1]{BenGreville} or the original~\cite[Theorem~1]{Penrose_1955}.

\begin{lemma}\label{lem:lowest_rank_solutions}
    \[WXW=\pinv{X}\Longleftrightarrow \rng{W^*}=\rng{X} \Longleftrightarrow P_{\nullsp{W}}YP_{\nullsp{W^*}}=0.\]
\end{lemma}

\begin{proof}
Not that for any $X$ we have $\rng{W^*}\subset\rng{X}$ by~(\ref{eq:riccati}) and the summands in~(\ref{eq:soln}) are pairwise perpendicular and the first spans $\rng{W^*}$ for any matrix $D$. This proves the second equivalence. The first condition holds iff the Penrose condition are satisfied, that is
\begin{enumerate}[i)]
    \item $X(WXW)X=X$,
    \item $(WXW)X(WXW)=WXW$,
    \item $X(WXW)=(XW)^2$ is Hermitian,
    \item $(WXW)X=(WX)^2$ is Hermitian.
\end{enumerate}
By Corollary~\ref{cor:equivalent} conditions $ii),iii)$ hold and condition $iv)$ hold as $(WXW)X(WXW)=(WX)^3W=WXW$. Condition $i)$ holds iff $X(WXW)=(XW)^2X=P_{\rng{W^*}}X=X$, that is, iff $\rng{X}\subset\rng{W^*}$. This proves the equivalence of the first and the second condition.
\end{proof}

\begin{lemma}\label{lem:identities}
    If $X=\pinv{B}\pinv{A}$ satisfies the equation
\[X(AB)(AB)^*(AB)X=\brac*{AB}^*,\]
then
\begin{enumerate}[i)]
    \item $AB\pinv{B}\pinv{A}AB=\pinv{\brac*{\pinv{B}\pinv{A}}}$,
    \item $\pinv{B}\pinv{A}AB\pinv{B}\pinv{A}=\pinv{\brac*{AB}}.$ 
\end{enumerate}
\end{lemma}

\begin{proof}
Let $W=AB$. Since, in general $\rng{W^*}\subset\rng{X}$, and $\rank{X}=\rank{W^*}$ by Lemma~\ref{lem:equal_ranks}, Lemma~\ref{lem:lowest_rank_solutions} implies claim $i)$. Claim $ii)$ follows from Corollary~\ref{cor:equivalent}.
\end{proof}

\begin{theorem}\label{thm:rev_and_riccati}
\[\pinv{\brac*{AB}}=\pinv{B}\pinv{A}\Longleftrightarrow\pinv{B}\pinv{A}(AB)(AB)^*(AB)\pinv{B}\pinv{A}=(AB)^*.\]
In such a case, the solution $X=\pinv{B}\pinv{A}$ is given by $D=\diag(I_r,0)$ and $Y=0$.
\end{theorem}

\begin{proof}
    $(\Longrightarrow)$ as in the introduction\\
    $(\Longleftarrow)$ by \cite[Theorem~11.1]{Tian512} point $\langle 37\rangle$ the reverse order law holds if and only if
    \[AB\pinv{B}\pinv{A}AB=\pinv{\brac*{\pinv{B}\pinv{A}}},\]
    which is exactly claim $i)$ from Lemma~\ref{lem:identities}
\end{proof}

\begin{remark}
    It seems that direct verification of the Penrose conditions for $\pinv{B}\pinv{A}$ and $AB$ is not straightforward.
\end{remark}

\begin{remark}
It can be shown using the Schur method, closely following the proof of Theorem~\ref{thm:main} that, in fact,
\[XWX=W^+ \Longleftrightarrow XWW^*WX=W^*.\]
\end{remark}

\begin{proof}
    (sketch) 
    Matrix\[M=\left[\begin{array}{c|c}
        0 & W\\ \hline
        \pinv{W} &  0\\
    \end{array}
    \right],\]
    has eigenvectors corresponding to the eigenvalues $\lambda=\pm 1$ of the form $w=\left[\begin{array}{c}
        u\\ \hline \pm \pinv{W}u
    \end{array}
    \right]\in\bK^{m+n}$, for any $u\in\rng{W}$ and eigenvectors  of the form $w=\left[\begin{array}{c}
        w_1 \\ \hline w_2
    \end{array}
    \right]\in\bK^{m+n}$,
    corresponding to the eigenvalue $0$ for any $w_1\in\nullsp{W^*}$ and $w_2\in\nullsp{W}$. This leads to an eigenspace of $M$ given by matrix~(\ref{eq:eigenspace}).
\end{proof}




\section{Uniqueness of Solutions}
\begin{lemma}\label{lem:uniqueness}
Assume that two solutions of~(\ref{eq:riccati}) in the form~(\ref{eq:soln}) are equal, that is 
\[V\pinv{\Sigma}D_1U^*+P_{\nullsp{W}} Y_1 P_{\nullsp{W^*}}=V\pinv{\Sigma}D_2U^*+P_{\nullsp{W}} Y_2 P_{\nullsp{W^*}}.\]    
Then 
\[D_1=D_2\quad\text{and}\quad P_{\nullsp{W}} Y_1 P_{\nullsp{W^*}}=P_{\nullsp{W}} Y_2 P_{\nullsp{W^*}}.\]

\end{lemma}
\begin{proof}
As
\[V\pinv{\Sigma}(D_1-D_2)U^*=P_{\nullsp{W}} (Y_2-Y_1) P_{\nullsp{W^*}}.\]
and $\nullsp{W}$ is perpendicular to $\rng{W^*}$ it follows that
\[V\pinv{\Sigma}(D_1-D_2)U^*=P_{\nullsp{W}} (Y_2-Y_1) P_{\nullsp{W^*}}=0.\]
In particular $D_1=D_2$.
    
\end{proof}
\begin{corollary}
Assume that $X=\pinv{B}\pinv{A}$ is a solution of~(\ref{eq:riccati}) for $W=AB$ and it is in the from~(\ref{eq:soln}). Then $D=\diag(I_r,0)$
and $P_{\nullsp{W}} Y P_{\nullsp{W^*}}=0$.
\end{corollary}
\begin{proof}
By Theorem~\ref{thm:rev_and_riccati}
    \[\pinv{B}\pinv{A}=\pinv{\brac*{AB}}=V\Sigma^+U^*.\]
\end{proof}

Since 
\[P_{\nullsp{W}}Y P_{\nullsp{W^*}}=0\Longleftrightarrow \brac*{P_{\nullsp{W^*}}^\intercal\otimes P_{\nullsp{W}}}\vect{Y}=0,\]
when $W$ has simple singular values, the set of solutions consists of $2^r$ disjoint families of dimension $mn-(m-r)(n-r)=r(m+n-r)$ over $\bK$.

If $W$ is Hermitian then $X$ is Hermitian for Hermitian matrices $Y$. In such case $X$ is positive (resp. negative) definite (if and only if $D=I$ (resp. $D=-I$) and $Y$ restricted to $\nullsp{W}$ is positive (resp. negative) definite.

\section{Stabilizing and Maximal Solutions}
In this Section we assume that $W=W^*$ is Hermitian matrix and hence diagonalizable. Let $W=U\Sigma U^*$.

\begin{lemma}
 \[X=X^*\Longleftrightarrow D=D^*\text{ and }P_{\nullsp{W}} \brac*{Y-Y^*} P_{\nullsp{W}}=0.\]   
\end{lemma}
 
 \begin{proof}
 
     Under the assumption, Equation~(\ref{eq:soln}) takes the form
     \[X=U\pinv{\Sigma}DU^*+P_{\nullsp{W}} Y P_{\nullsp{W}}.\]
     If $X=X^*$ then by Lemma~\ref{lem:uniqueness}
     \[\pinv{\Sigma} D=D^*\brac*{\pinv{\Sigma}}^*.\]
     Since $D=\diag(X_1,\ldots,X_t,0,\ldots,0)$, by  considering each block
      this gives $X_i=X_i^*$ for $i=1,\ldots,t$, that is, $D=D^*$. In particular, each $X_i$ is an orthogonal symmetry. The other direction follows by inspection.
 \end{proof}

For the following definition, cf.~\cite[Definition~16.16.12,15.9.1]{Bernstein}.
\begin{definition}
A Hermitian solution $X_{max}$ of~(\ref{eq:riccati}) is called the maximal solution if $X\preceq X_{max}$ for any other Hermitian solution $X$. A Hermitian solution $X$ of~(\ref{eq:riccati}) called the stabilizing solution if $-W^3X$ is asymptotically stable. A matrix $A\in\bC^{n\times n}$ is called asymptotically stable if $\Re\lambda<0$ for any eigenvalue $\lambda$ of $A$.
\end{definition}    

\begin{lemma}
    Assume that $W$ is an invertible Hermitian matrix. Then $X=\pinv{W}=W^{-1}$ is the maximal solution and simultaneously the unique stabilizing solution.
\end{lemma}

\begin{proof}
    Since $P_{\nullsp{W}}=0$ any Hermitian solution of~(\ref{eq:riccati}) is of the form
    \[X=U\pinv{\Sigma}DU^*,\]
    for some invertible Hermitian and block diagonal matrix $D$ such that $D^2=I$. The eigenvalues of $D$ are equal to $1$ or $-1$ therefore $D\preceq I$ and consequently \[X=U\pinv{\Sigma}DU^*\preceq U\pinv{\Sigma}IU^*=\pinv{W}.\]
    Moreover $-W^3X=-W^2=-U\Sigma^2U^*$ and all its eigenvalues are negative.
\end{proof}

\section{Examples}
 The following code illustrates the correctedness of Theorem~\ref{thm:main} and shows that if $W^*=XWW^*WX$ then $(WX)^3=WX$ but possibly $(WX)^2\neq WX$.

\begin{lstlisting}
m = 10; n = 15; k = 7; r=k; % k<=min(m,n) 
W = rand(m,k)*rand(k,n);
PM = 2*randi([0 1], 1, r)-1; % random +1 or -1's
D = [ diag(PM) zeros(r,m-r);
      zeros(m-r,m)];
% generic W has simple singular values hence D is diagonal
[U,S,V] = svd(W);
X = V*pinv(S)*D*U'+(eye(n)-pinv(W)*W)*rand(n,m)*(eye(m)-W*pinv(W));
norm(X*W*W'*W*X-W') + norm(X*W*X-pinv(W)) + norm((X*W)^3-X*W) + norm((W*X)^3-W*X)
norm((W*X)^2-W*X)
norm((X*W)^2-X*W)
\end{lstlisting}
 
\begin{example}
    Let $W$ be a square unitary matrix. Let $W=U\Sigma V^*$ be an SVD decomposition for $U=W$ and $\Sigma=V=I$. Then $X=DW^*$ solves
    (\ref{eq:riccati}) for any square matrix $D$ such that $D^2=1$. In particular, $XW=D$ in general is not Hermitian.
\end{example}

The following code shows that, in general, $\rng{\pinv{B}\pinv{A}}\neq\rng{\pinv{\brac*{AB}}}$.
\begin{lstlisting}
A=rand(5,2)*rand(2,9); B=rand(9,3)*rand(3,8); 
X=pinv(B)*pinv(A); 
% different
rref(X')
rref(A*B)
\end{lstlisting}
In addition, in general $\rng{\pinv{B}\pinv{A}}=\rng{\pinv{\brac*{AB}}}\centernot\Longrightarrow \pinv{B}\pinv{A}=\pinv{\brac*{AB}}$, for example when $A^*$ is surjective, then the antecedent hold for any compatible matrix $B$. The converse obviously holds. Conditions equivalent to $\rng{\pinv{B}\pinv{A}}=\rng{\pinv{\brac*{AB}}}$ are discussed in~\cite[Lemma~11.3]{Tian512}, cf. point $\langle 26\rangle$.
\begin{lstlisting}
A=rand(10,9); B=rand(9,2)*rand(2,5); 
X=pinv(B)*pinv(A); norm(X-pinv(A*B))
rref(X')
rref(A*B)

\end{lstlisting}

\bibliographystyle{unsrt}  
\bibliography{references}  

\end{document}